\documentclass[a4paper,12pt]{amsart}
\usepackage{amssymb,amscd}
\usepackage[cp1251]{inputenc}
\usepackage[T2A]{fontenc}
\usepackage{amsmath}
\usepackage{amsthm}
\usepackage{mathtext}
\usepackage{euscript}

\usepackage{graphicx}
\graphicspath{}
\DeclareGraphicsExtensions{.pdf,.png,.jpg,.jpeg}

\linethickness{0.5pt}

\usepackage[english]{babel}

\newtheorem{theo}{Theorem}[section]
\newtheorem{prop}[theo]{Proposition}
\newtheorem{lemm}[theo]{Lemma}
\newtheorem{coro}[theo]{Corollary}

\theoremstyle{definition}
\newtheorem{defi}{Definition}
\newtheorem{exam}[theo]{Example}
\newtheorem{prob}[theo]{Problem}
\newtheorem{constr}[theo]{Construction}

\theoremstyle{remark}
\newtheorem*{rema}{Remark}

\numberwithin{equation}{section}

\newcommand{\mb}[1]{{\textbf {\textit#1}}}

\newcommand{\field}[1]{\mathbb{#1}}
\def\C{\field{C}}

\newcommand{\R}{\field{R}}

\DeclareMathOperator{\Tor}{Tor}
\DeclareMathOperator{\ko}{k}

\DeclareMathOperator{\conv}{conv}

\DeclareMathOperator{\MF}{MF}

\def\ge{\geqslant}

\newcommand{\zp}{\mathcal Z_P}

\newcommand{\zk}{\mathcal Z_K}

\begin{document}

\title[DIRECT FAMILIES OF POLYTOPES]{DIRECT FAMILIES OF POLYTOPES\\ WITH NONTRIVIAL MASSEY PRODUCTS}

\author{Victor Buchstaber}
\address{Department of Mathematics and Mechanics, Moscow
State University, 1 Leninskie Gory, Moscow, Russia}
\address{Steklov Mathematical Institute of the Russian Academy of Sciences, 8 Gubkina street 8, Moscow, Russia}
\email{buchstab@mi.ras.ru}

\author{Ivan Limonchenko}
\address{School of Mathematical Sciences, Fudan University, 220 Handan Road, Shanghai, P.R. China}
\email{ilimonchenko@fudan.edu.cn}

\thanks{The first author was supported by the Russian Foundation for Basic Research, grants nn.~16-51-55017 and~17-01-00671. The second author was supported by the General Financial Grant from the China Postdoctoral Science Foundation, grant no. 2016M601486.}



\maketitle

\begin{abstract}
The problem of existence of nontrivial Massey products in cohomology of a space is well-known in algebraic topology and homological algebra. A number of problems in complex geometry, symplectic geometry and algebraic topology can be stated in terms of Massey products. One of such problems is to establish formality of smooth manifolds in rational homotopy theory. There have already been constructed a few classes of spaces with nontrivial triple Massey products in cohomology. Until now, very few examples of manifolds $M$ with nontrivial higher Massey products in $H^*(M)$ were known. In this work we introduce a sequence of smooth closed manifolds $\{M_{k}\}^{\infty}_{k=1}$ such that $M_{k}\hookrightarrow M_{k+1}$ is a submanifold and a retract of $M_{k+1}$ for any $k\geq 1$ and there exists a nontrivial Massey product $\langle\alpha_{1},\ldots,\alpha_{n}\rangle$ in $H^*(M_{k})$ for each $2\leq n\leq k$. The sequence $\{M_k\}^{\infty}_{k=1}$ is determined by a new family of flag nestohedra $\mathcal P_{Mas}$. We give P.D.E. for the two-parametric generating series of $\mathcal P_{Mas}$. 
\end{abstract}

\section{Introduction and main results}

The aim of this paper is to develop a theory of families of polytopes $\mathcal P=\{P^n|\,n\geq 0\}$ such that there exists a nontrivial $k$-fold Massey product in cohomology of a moment-angle manifold $\mathcal Z_{P^n}$, $k\to\infty$ as $n\to\infty$. Namely, we construct a family of flag nestohedra $\mathcal P_{Mas}=\{P^{n}|\,n\geq 0\}$, where $P^0$ is a point, $P^1$ is a segment, $P^2$ is a square, and $P^n, n>2$ is such that there exists a nontrivial strictly defined Massey product $\langle\alpha_{1},\ldots,\alpha_{k}\rangle$ in $H^*(\mathcal Z_{P^n})$ with $\dim\alpha_{i}=3$ for each $2\leq k\leq n$. We introduce new notions for families of polytopes: a direct family of polytopes (DFP) and a direct family of polytopes with nontrivial Massey products (DFPM). Our main result is that the family $\mathcal P_{Mas}$ is DFP and DFPM.

Obviously, any defined $k$-fold Massey product in $H^*(\zp)$ is trivial for $k>2$ when $P=I^n$ and is trivial for $k>1$ when $P=\Delta^n$.
The first construction of a nontrivial Massey product in cohomology of a moment-angle-complex $\zk$ was given in the work of Baskakov~\cite{BaskM} (2003) (see details in~\cite[Construction 4.9.1]{TT}). It was shown that there exists a class of triangulated spheres $K$ obtained from a join of three triangulated spheres of arbitrary dimensions by performing two certain stellar subdivisions such that there is a defined and nontrivial triple Massey product $\langle\alpha_{1},\alpha_{2},\alpha_{3}\rangle$ in $H^*(\zk)$.   

Denham and Suciu~\cite{DS} (2007) proved necessary and sufficient conditions on combinatorics of a simplicial complex $K$ under which there exists a nontrivial triple Massey product of 3-dimensional cohomology classes $\alpha_1,\alpha_2,\alpha_3$ of $\zk$.

The following definition plays a crucial role in what follows.
A \emph{family of polytopes} $\mathcal F$ is a set of polytopes such that for any $n\geq 0$ the subset $\mathcal F_n$ of its $n$-dimensional elements is nonempty and finite. A family $\mathcal F$ is called a \emph{line family} if each $\mathcal F_i$ consists of exactly one polytope.
Note that $\mathcal F_{0}=\{pt\}, \mathcal F_{1}=\{I\}$ for any family of polytopes $\mathcal F$. 

It was proved in~\cite[Lemma 2.13]{B-L} that if $P$ is flag and $F\subset P$ is its face then the maps $\hat{i}_{F}:\mathcal Z_{F}\to\mathcal Z_P$ and $\hat{j}_{F,P}:\mathcal Z_{K_F}\to\mathcal Z_{K_P}$, induced by embeddings of a face $i_{F,P}:\,F\to P$ and that of a full subcomplex $K_F\to K_P$, are linked by canonical homeomorphisms $\mathcal Z_F\to\mathcal Z_P$ and $\mathcal Z_{K_F}\to\mathcal Z_{K_P}$ into a commutative diagram. One of our key results is that the same holds for any nestohedron (not necessarily flag), see Definition~\ref{nest}. Namely, we show that when $P=P_B$ is a nestohedron, the face embedding $i_{P_{B|_S}}$ and the embedding of full subcomplexes $j_{P_{B|_S}}$ induce mappings of moment-angle manifolds and complexes that fit in a commutative diagram, see Proposition~\ref{phiPsi}.

In the work of Buchstaber~\cite{B} (2008) the bigraded ring of simple polytopes $(\mathcal P,d)$ was constructed, where the differential $d$ sends a polytope $P$ to a disjoint union of its facets. In~\cite{B} and the papers by Buchstaber and Volodin~\cite{BV, BV2} partial differential equations were obtained such that their solutions are 2-parametric generating series of well-known line families of nestohedra $\{P^n|\,n\geq 0\}$, where $P^n$ is an $n$-dimensional simplex $\Delta^n$, cube $I^n$, associahedron $As^n$, cyclohedron $Cy^n$, permutohedron $Pe^n$, or a stellahedron $St^n$. Such a P.D.E. carries important information on how rich is the combinatorics of the family. 

In this work in terms of the operator $d$ on $\mathcal P$ we introduce the notion of {\emph{complexity}} of a family $\mathcal F$. Namely, $\mathcal F$ has complexity $k$ if its $d$-closure is a union of $k$ line families, see Definition~\ref{complex}. In particular, associahedra and permutohedra families have complexity 1, stellahedra and cyclohedra families have complexity 2.

Using~\cite[Theorem 6.1.1]{DS}, Limonchenko~\cite{L2} proved that for any graph-associahedron $P\neq Pe^3$ there exists a nontrivial triple Massey product $\langle\alpha_{1},\alpha_{2},\alpha_{3}\rangle$ with $\dim\alpha_{i}=3$ in $H^*(\zp)$ and any defined triple Massey product of 3-dimensional cohomology classes for $\mathcal Z_{Pe^3}$ is trivial. 
In~\cite{L3} he constructed a nontrivial triple Massey product $\langle\alpha_{1},\alpha_{2},\alpha_{3}\rangle$ in $H^*(\mathcal Z_{Pe^3})$, where $\dim\alpha_2=3$ and $\dim\alpha_1=\dim\alpha_3=5$.
 
First examples of moment-angle manifolds $\zp$ with higher nontrivial Massey products $\langle\alpha_{1},\ldots,\alpha_{k}\rangle$ for any $k\geq 4$ in $H^*(\zp)$ were constructed by Limonchenko~\cite{L1,L2}. He introduced a family $\mathcal Q=\{Q^n|\,n\geq 0\}$ such that there exists a nontrivial Massey product $\langle\alpha_{1},\ldots,\alpha_{n}\rangle$ of order $n$ with $\dim\alpha_{i}=3,1\leq i\leq n$ in $H^*(\mathcal Z_{Q^n})$ for any $n\geq 2$. Recently, Limonchenko also introduced a family of moment-angle manifolds of arbitrarily large connectedness having nontrivial Massey products $\langle\alpha_{1},\ldots,\alpha_{k}\rangle$ in cohomology for any $k\geq 3$, see~\cite{L3}.

It was proved in~\cite{PRW06} that if $P_B$ is a flag polytope, then there exists a building set $B_{0}\subseteq B$ such that $P_{B_0}$ is a combinatorial cube with $\dim P_{B_0}=\dim P_{B}$. Therefore, the class of polytopes obtained from cubes by truncation of faces contains all flag nestohedra, in particular, all graph-associahedra. Moreover, it was proved in~\cite{BV}, that a nestohedron is flag if and only if it can be realized as a \emph{2-truncated cube}, that is a polytope obtained from a cube as a result of a sequence of cutting off faces of codimension 2 only. The polytopes $Q^n$ provide us with an important family $\mathcal Q$ of 2-truncated cubes, which are not nestohedra for $n\geq 4$. 

Among the new notions that we introduce in this paper, the following two are the main ones. We denote by $m(P^n)$ the number of facets of an $n$-dimensional polytope $P$. In what follows, when the choice of $P$ is clear, we write $m(n)$, or even $m$, instead of $m(P^n)$ to simplify our notation.

\begin{defi}\label{ADFP}
A sequence of polytopes $\mathcal F=\{P^n|\,n\geq 0\}$ is called an \emph{algebraic direct family} (ADFP) if the following conditions hold:
\begin{itemize}
\item[(1)] For any $r$ and $n>r$ there exists a face $F^r$ of $P^n$ such that $F^{r}$ is combinatorially equivalent to $P^r$ and $\{P^n, i_{r}^{n}\}$ forms a direct system of polytopes with respect to face embeddings $i_{r}^{n}:\,P^{r}\hookrightarrow\partial P^n\subset P^n$ (we identify $F^{r}$ with $P^r$); 
\item[(2)] For any $r$ and $n>r$ there exists a vertex set $J\subset [m(n)]=\{1,2,\ldots,m(n)\}$ such that the full subcomplex of $K_{P^n}$ on $J$ is combinatorially equivalent to $K_{P^r}$ and $\{K_{P^n}, j_{r}^{n}\}$ forms a direct system of triangulated spheres with respect to simplicial embeddings $j_{r}^{n}:\,K_{P^{r}}\hookrightarrow K_{P^n}$ induced by $J\subset [m(n)]$ (we identify $(K_{P^n})_J$ with $K_{P^r}$).
\end{itemize}
\end{defi}

Note that if $\mathcal F$ consists of flag polytopes, then (1) implies (2) due to~\cite[Theorem 2.16]{B-L}. Note also that (2) implies that $\{H^*(\mathcal Z_{P^n}), (j_{r}^{n})^{*}\}$ forms an inverse system of rings with respect to split ring epimorphisms $(j_{r}^{n})^*$, see~\cite[Corollary 2.7]{B-L}, which motivates the name 'algebraic' in the definition above. Each $i_{r}^n$ also induces a map of moment-angle manifolds $\hat{\phi}_{r}^{n}:\,\mathcal Z_{P^r}\rightarrow\mathcal Z_{P^n}$, see \cite[Construction 2.8]{B-L}. 

\begin{defi}\label{GDFP}
Suppose $\mathcal F$ satisfies the conditions (1) and (2) of the above definition alongside with the following addition condition:
\begin{itemize}
\item[(3)] $(\hat{\phi}_{r}^{n})^{*}$ and $(j_{r}^{n})^{*}$
are equivalent as group homomorphisms.
\end{itemize}
It follows that $(\hat{\phi}_{r}^{n})^{*}:\,H^*(\mathcal Z_{P^n})\rightarrow H^*(\mathcal Z_{P^r})$ are also split ring epimorphisms. In this case we call $\mathcal F$ a \emph{geometric direct family} (GDFP).
\end{defi}

Observe that, by the above argument, any ADFP consisting of flag polytopes is a GDFP. One of the purposes of this work is to construct examples of geometric direct families of (not necessarily flag) polytopes possessing nice combinatorial properties. This problem is solved by means of nestohedra theory.

\begin{defi}\label{DFPM}
A direct family of polytopes $\mathcal F=\{P^n|\,n\geq 0\}$ is called \emph{a direct family with nontrivial Massey products} (DFPM) if there exists a nontrivial Massey product of order $k$ in $H^*(\mathcal Z_{P^n})$ with $k\to\infty$ as $n\to\infty$.
\end{defi}

Note that if $\mathcal F$ satisfies conditions (1) and (2) of Definition~\ref{ADFP}, the following property immediately follows from~\cite[Corollary 2.7]{B-L} and Theorem~\ref{BPtheo} below. If $\langle\alpha_{1}^{r},\ldots,\alpha_{k}^{r}\rangle$ is a defined Massey product in $H^*(\mathcal Z_{P^r})$, then for each $n>r$ there exists a defined Massey product $\langle\alpha_{1}^{n},\ldots,\alpha_{k}^{n}\rangle$ in $H^*(\mathcal Z_{P^n})$ such that $(j_{r}^{n})^{*}(\alpha_{t}^{n})=\alpha_{t}^{r}$ for all $1\leq t\leq k$ and 
$$
(j_{r}^{n})^{*}\langle\alpha_{1}^{n},\ldots,\alpha_{k}^{n}\rangle=\langle\alpha_{1}^{r},\ldots,\alpha_{k}^{r}\rangle.
$$

By~\cite[Theorem 3.2]{B-L}, the family of 2-truncated cubes $\mathcal Q=\{Q^n|\,n\geq 0\}$ introduced in~\cite{L1} is a geometric direct family of polytopes with nontrivial Massey products.

If $\mathcal F$ satisfies conditions (1)-(3) of Definition~\ref{GDFP}, the following property takes place by~\cite[Lemma 2.13]{B-L}. If $\langle\alpha_{1}^{r},\ldots,\alpha_{k}^{r}\rangle$ is a defined Massey product in $H^*(\mathcal Z_{P^r})$, then for each $n>r$ there exists a defined Massey product $\langle\alpha_{1}^{n},\ldots,\alpha_{k}^{n}\rangle$ in $H^*(\mathcal Z_{P^n})$ such that $(\hat{\phi}_{r}^{n})^{*}(\alpha_{t}^{n})=\alpha_{t}^{r}$ for all $1\leq t\leq k$ and 
$$
(\hat{\phi}_{r}^{n})^{*}\langle\alpha_{1}^{n},\ldots,\alpha_{k}^{n}\rangle=\langle\alpha_{1}^{r},\ldots,\alpha_{k}^{r}\rangle.
$$

We show that the well-known families of nestohedra $\mathcal P=\{P^n|\,n\geq 0\}$ (except for simplices) mentioned above are geometric direct families of polytopes. 

{\bf{Problem I.}}
Are their direct families with nontrivial Massey products satisfying conditions (1)-(3) above among the families of associahedra $As$, cyclohedra $Cy$, stellahedra $St$, or permutohedra $Pe$?

A sequence of connected graphs $\Gamma=\{\Gamma_{n}|\,n\geq 0,\Gamma_{n}\,\text{is on }[n+1]=\{1,2,\ldots,n+1\}\}$ is called a \emph{direct family of graphs} (DFG) if for any $r$ and $n>r$ there exists a subset of vertices $S\subset [n+1]$ such that the indused subgraph $\Gamma_{n}|_{S}$ is isomorphic to $\Gamma_{r}$ and $\{\Gamma_{n}, k_{r}^{n}\}$ forms a direct system of graphs with respect to graph embeddings $k_{r}^{n}:\,\Gamma_{r}\hookrightarrow\Gamma_{n}$ induced by $S\subset [n+1]$ (we identify $\Gamma_{n}|_S$ with $\Gamma_{r}$).

Observe that the families of graphs of associahedra, cyclohedra, stellahedra, and permutohedra all are direct families of graphs. It follows from the above definitions that $\mathcal P_{\Gamma}=\{P_{\Gamma_n}|\,n\geq 0\}$ is a geometric direct family of polytopes for any direct family of graphs $\Gamma$. Then it is natural to state the following problem.

{\bf{Problem II.}}
Is there a sequence of graphs $\{\Gamma_{n}|\,n\geq 0\}$ such that there exists a nontrivial $k$-fold Massey product in cohomology of a moment-angle manifold $\mathcal Z_{\Gamma_n}$ with $k\to\infty$ as $n\to\infty$?

The family of 2-truncated cubes $\mathcal Q$ was our starting point in constructing a geometric direct family $\mathcal P_{Mas}$ of flag nestohedra with nontrivial Massey products, which we describe in Definition~\ref{familyP} below. The polytopes $P^n_{Mas}\in\mathcal P_{Mas}$ are not graph-associahedra for $n\geq 3$, therefore, it is natural to state the following problem.

{\bf{Problem III.}}
Is there a DFG $\Gamma$ such that the corresponding family of graph-associahedra $\mathcal P_{\Gamma}$ is a direct family with nontrivial Massey products?

We call a direct family $\mathcal F=\{P^n|\,n\geq 0\}$ of polytopes with nontrivial Massey products {\emph{special}} if for any $n\geq 2$ there exists a strictly defined nontrivial Massey product $\langle\alpha_{1},\ldots,\alpha_{k}\rangle$ in $H^*(\mathcal Z_{P^n})$ for each $2\leq k\leq n$.

Our main result, see Theorem~\ref{MasAllOrder} and Theorem~\ref{DirFamNontrivial}, is that the family $\mathcal P_{Mas}$ is a special geometric direct family of flag nestohedra with nontrivial Massey products, having an additional condition that all $\dim\alpha_{i}=3$. 

Families of nestohedra that we constructed in this paper all have complexity not greater than 4. Therefore, the following problem naturally arises.

{\bf{Problem IV.}}
To describe and construct families of polytopes with finite complexity.

\section{Basic notions and motivation}

Here we recall only the notions which are the key ones for our constructions and results. For the definitions, constructions, and results on other notions that we are using in this work, we refer the reader to the monograph~\cite{TT}. 

\begin{defi}
An \emph{(abstract) simplicial complex} $K$ on the vertex set $[m]=\{1,2,\ldots,m\}$ is a set of subsets of $[m]$ called its \emph{simplices} such that:
\begin{itemize}
\item[(1)] The empty set $\varnothing$ is a simplex of $K$;
\item[(2)] If $\tau\subset\sigma$ and $\sigma\in K$, then $\tau\in K$.
\end{itemize}
\end{defi}

In what follows we assume, unless otherwise is stated explicitly, that there are no \emph{ghost vertices} in $K$, that is $\{i\}\in K$ for all $i\in [m]$.
We call a {\emph{full subcomplex}} $K_J$ on the vertex set $J\subseteq [m]$ the set of all elements in $K$ with vertices in $J$, that is $K_{J}=2^{J}\cap K$.

\begin{constr}(moment-angle-complex)\label{mac}
Suppose $K$ is a simplicial complex on $[m]$. Then define its \emph{moment-angle-complex} $\zk$ to be
$$
\zk=\cup_{\sigma\in K} (D^2,S^1)^{\sigma}, 
$$
where $(D^2,S^1)^{\sigma}=\prod\limits_{i=1}^{m}\,Y_{i}$ such that 
$Y_{i}=D^2$, if $i\in\sigma$, and $Y_{i}=S^1$, otherwise.

Buchstaber and Panov, see~\cite[Chapter 4]{TT}, proved that $\zk$ is a CW complex for any $K$. 
\end{constr}

In our paper we use the following definition of a simple convex polytope. 

\begin{defi}\label{Simplepolytopes}
A \emph{simple convex $n$-dimensional polytope} $P$ in
the Euclidean space $\R^n$ with scalar product
$\langle\;,\:\rangle$ can be defined as a bounded
intersection of $m$ halfspaces:
 $$
  P=\bigl\{\mb x\in\R^n\colon\langle\mb a_i,\mb
  x\rangle+b_i\ge0\quad\text{for }
  i=1,\ldots,m\bigr\},\eqno (1)
$$
where $\mb a_i\in\R^n$, $b_i\in\R$. We assume that its \emph{facets}  
$$
  F_i=\bigl\{\mb x\in P\colon\langle\mb a_i,\mb
  x\rangle+b_i=0\bigr\},\quad\text{for } i=1,\ldots,m.
$$
are in general position and that there are no redundant inequalities
in $(1)$, 
that is, no inequality can be removed
from $(1)$ 
without changing~$P$. The latter is equivalent to the condition that its \emph{nerve complex} $K_P=\partial P^*$ has no ghost vertices.
\end{defi}

\begin{constr}\label{simpmultwedge}(\cite{BBCG15})
Let $J=(j_{1},\ldots,j_{m})$ be an $m$-tuple of positive integers. Consider the following vertex set:
$$
m(J)=\{11,\ldots,1j_{1},\ldots,m1,\ldots,mj_{m}\}.
$$
To define \emph{a simplicial multiwedge, or a $J$-construction} of $K$, which is a simplicial complex $K(J)$ on $m(J)$, we say that the set of its minimal non-faces $\MF(K(J))\subset 2^{[m(J)]}$ consists of the subsets of $m(J)$ of the type 
$$
I(J)=\{i_{1}1,\ldots,i_{1}j_{i_1},\ldots,i_{k}1,\ldots,i_{k}j_{i_k}\},
$$
where $I=\{i_{1},\ldots,i_{k}\}\in\MF(K)$. Note that if $J=(1,\ldots,1)$, then $K(J)=K$. 

Suppose $K=K_P$ is a polytopal sphere on $m$ vertices. Then $K(J)$ is always a polytopal sphere and thus a nerve complex of a simple polytope $Q=P(J)$, and so $K_{P(J)}=K_{P}(J)$.
\end{constr}

The next construction appeared firstly in the work of Davis and Januszkiewicz~\cite{DJ}.

\begin{constr}\label{mamfdDJ}
Suppose $P^n$ is a simple convex polytope with the set of facets $F=\{F_{1},\ldots,F_{m}\}$. 
Denote by $T^{F_{i}}$ a 1-dimensional coordinate subgroup in $T^{F}\cong T^{m}$ for each $1\leq i\leq m$ and $T^{G}=\prod\,T^{F_i}\subset T^{F}$ for a face $G=\cap\,F_{i}$ of a polytope $P^n$. Then the \emph{moment-angle manifold} $\zp$ over~$P$ is defined as a quotient space
$$
\zp=T^{F}\times P^{n}/\sim,
$$
where $(t_{1},p)\sim (t_{2},q)$ if and only if $p=q\in P$ and $t_{1}t_{2}^{-1}\in T^{G(p)}$, $G(p)$ is a minimal face of $P$ which contains $p=q$. 
\end{constr}

\begin{rema} 
It can be deduced from Construction~\ref{mamfdDJ} that if $P_{1}$ and $P_{2}$ are \emph{combinatorially equivalent}, i.e. their face lattices are isomorphic, then $\mathcal Z_{P_{1}}$ and $\mathcal Z_{P_{2}}$ are homeomorphic. The opposite statement is {\emph{not}} true. Moreover, $\mathcal Z_{P_{1}\times P_2}$ is homeomorphic to $\mathcal Z_{P_1}\times\mathcal Z_{P_2}$.
\end{rema}


Let $A_P$ be the $m\times n$ matrix of row vectors $\mb a_i\in\mathbb{R}^n$, and
let $\mb b_P\in\mathbb{R}^m$ be the column vector of scalars $b_i\in\R$. Then we
can rewrite $(*)$ 
as
\[
  P=\bigl\{\mb x\in\R^n\colon A_P\mb x+\mb b_P\ge\mathbf 0\},
\]
and consider the affine map
\[
  i_P\colon \R^n\to\R^m,\quad i_P(\mb x)=A_P\mb x+\mb b_P.
\]
It embeds $P$ into
\[
  \R^m_\ge=\{\mb y\in\R^m\colon y_i\ge0\quad\text{for }
  i=1,\ldots,m\}.
\]

Moment-angle manifolds have already become key objects of study in Toric Topology, see~\cite{TT}. Buchstaber and Panov gave the following definition and proved it to be equivalent to the one above.
 
\begin{constr}\label{mamfdBP}
Define a \emph{moment-angle manifold} $\mathcal Z_P$ of a polytope $P$
as a pullback from the commutative diagram
$$\begin{CD}
  \mathcal Z_P @>i_Z>>\C^m\\
  @VVV\hspace{-0.2em} @VV\mu V @.\\
  P @>i_P>> \R^m_\ge
\end{CD}\eqno 
$$
where $\mu(z_1,\ldots,z_m)=(|z_1|^2,\ldots,|z_m|^2)$. The projection $\zp\rightarrow P$ in the above diagram is the quotient map of the canonical action of the compact torus $\mathbb{T}^m$ on $\zp$ induced by the standard action of $\mathbb{T}^m$
\[
  \mathbb T^m=\{\mb z\in\C^m\colon|z_i|=1\quad\text{for }i=1,\ldots,m\}
\]
on~$\C^m$. Therefore, $\mathbb T^m$ acts on $\zp$ with an orbit space $P$, and $i_Z$ is a $\mathbb T^m$-equivariant embedding.
\end{constr}

It follows immediately from the Construction~\ref{mamfdBP}, see~\cite[\S3]{BR}, that $\zp$ is a total intersection of Hermitian quadrics in $\C^m$. Thus, $\zp$ obtains a canonical equivariant smooth structure. For any simple $n$-dimensional polytope $P$ with $m$ facets its moment-angle manifold $\zp$ is 2-connected closed $(m+n)$ dimensional manifold. Furthermore, in~\cite{bu-pa00-2} a canonical equivariant homeomorphism $h_{P}:\,\mathcal Z_{P}\to\mathcal Z_{K_P}$ for any simple polytope $P$ was constructed. 

Now we recall a description of the cohomology algebra of $\zp$ and higher Massey products in it. 

Let $\ko$ be a commutative ring with a unit. Throughout the paper, unless otherwise is stated explicitly, we denote by $K$ an $(n-1)$-dimensional simplicial complex on the vertex set $[m]=\{1,2,\ldots,m\}$ and by $P$ a simple convex $n$-dimensional polytope with $m=m(n)$ facets: $\mathcal F(P)=\{F_{1},\ldots,F_{m}\}$. 

\begin{defi}\label{Facerings}
A \emph{face ring} (or a \emph{Stanley-Reisner ring}) of $K$ is the quotient ring
$$
   \ko[K]:=\ko[v_{1},\ldots,v_{m}]/I_K
$$
where $\deg v_{i}=2$ and $I_K$ is the ideal generated by square free
monomials $v_{i_{1}}\cdots{v_{i_{k}}}$ such that $\{i_{1},\ldots,i_{k}\}\notin K$.\\
We call a \emph{face ring of a polytope} $P$ the Stanley-Reisner ring of its nerve complex: $\ko[P]=\ko[K_P]$.
\end{defi}

Note that $\ko[P]$ is a module over $\ko[v_{1},\ldots,{v_{m}}]$ via
the quotient projection. 
The following result relates the cohomology algebra of $\zp$ to combinatorics of $P$.

\begin{theo}[{\cite[Theorem 4.5.4]{TT} or \cite[Theorem 4.7]{P}}]\label{BPtheo}
The following statements hold.
\begin{itemize}
\item[(I)] The isomorphisms of algebras hold:
$$
\begin{aligned}
  H^*(\zp;\ko)&\cong\Tor_{\ko[v_1,\ldots,v_m]}^{*,*}(\ko[P],\ko)\\
  &\cong H^{*,*}\bigl[\Lambda[u_1,\ldots,u_m]\otimes \ko[P],d\bigr]\\
  &\cong \bigoplus\limits_{J\subset [m]}\widetilde{H}^*(P_{J};\ko),
\end{aligned}
$$
$\mathop{\mathrm{bideg}} u_i=(-1,2),\;\mathop{\mathrm{bideg}} v_i=(0,2);\quad
  du_i=v_i,\;dv_i=0.$
Here we denote $P_J=\cup_{j\in J}\,F_{j}$.
The last isomorphism is the sum of isomorphisms of $\ko$-modules:
$$
H^p(\zp;\ko)\cong\sum\limits_{J\subset [m]}\widetilde{H}^{p-|J|-1}(P_{J};\ko);
$$

\item[(II)] If we define a finitely generated differential graded algebra $R(P)=\Lambda[u_{1},\ldots,u_{m}]\otimes\ko[P]/(v_{i}^{2}=u_{i}v_{i}=0,1\leq i\leq m)$ with the same $d$ as above, then a graded algebra isomorphism holds: 
$$
H^{*,*}(\zp;\ko)\cong H^{*,*}[R(P),d]\cong\Tor^{*,*}_{\ko[v_{1},\ldots,v_{m}]}(\ko[P],\ko).
$$
These algebras admit $\mathbb{N}\oplus\mathbb{Z}^m$-multigrading ($i\in\mathbb{N}$ and $J\in\mathbb{Z}^m$ below), which is respected by $d$, and we have
$$
\Tor^{-i,2{\bf{a}}}_{\ko[v_{1},\ldots,v_{m}]}(\ko[P],\ko)\cong H^{-i,2{\bf{a}}}[R(P),d], 
$$
where $\Tor^{-i,2J}_{\ko[v_{1},\ldots,v_{m}]}(\ko[P],\ko)\cong\widetilde{H}^{|J|-i-1}(P_{J};\ko)$
for $J\subseteq [m]$. The multigraded component $\Tor^{-i,2{\bf{a}}}_{\ko[v_{1},\ldots,v_{m}]}(\ko[P],\ko)=0$, if ${\bf{a}}$ is not a $(0,1)$-vector of length $m$. 
\end{itemize}
\end{theo}

Now we recall one of the key definitions, namely, Massey operations in cohomology of a differential graded algebra. Our exposition follows that in~\cite[Appendix $\Gamma$]{BP04}; see also the work of Babenko and Taimanov~\cite{BaTa}.

\begin{defi}\label{DefiningSystem}
Suppose $(A,d)$ is a differential graded algebra, $\alpha_{i}=[a_{i}]\in H^{*}[A,d]$ and $a_{i}\in A^{n_{i}}$ for $1\leq i\leq k$.
Then a \emph{defining system} for $(\alpha_{1},\ldots,\alpha_{k})$ is a $(k+1)\times (k+1)$-matrix $C$ such that the following conditions hold:
\begin{itemize}
\item[{(1)}] $c_{i,j}=0$, if $i\geq j$,
\item[{(2)}] $c_{i,i+1}=a_{i}$,
\item[{(3)}] $a\cdot E_{1,k+1}=dC-\bar{C}\cdot C$ for some $a=a(C)\in A$, where $\bar{c}_{i,j}=(-1)^{deg{c_{i,j}}}\cdot c_{i,j}$ and $E_{1,k+1}$ is a $(k+1)\times (k+1)$-matrix with all elements equal to zero, except for that in the position $(1,k+1)$, which equals 1.
\end{itemize} 

It follows that $d(a)=0$ and $a=a(C)\in A^{m}$, $m=n_{1}+\ldots+n_{k}-k+2$. Therefore, a cohomology class $[a(C)]\in H^{m}[A,d]$ is defined.

A \emph{Massey product} $\langle\alpha_{1},\ldots,\alpha_{k}\rangle$ of order $k$ is said to be \emph{defined}, if there exists a defining system $C$ for it; in this case, by definition, 
$$
\langle\alpha_{1},\ldots,\alpha_{k}\rangle=\{\alpha=[a(C)]|\,C\text{ is a defining system}\}.
$$
A defined Massey product is called \emph{trivial}, or \emph{vanishing} if $[a(C)]=0$ for some defining system $C$. We refer to a defined Massey product of order $k>3$ as a {\emph{higher}} Massey product. 
\end{defi}

We end this section with a few crucial definitions that belong to the theory of polytopes. Following Feichtner--Sturmfels~\cite{FS} and Postnikov~\cite{Post05}, we now give the definitions of a building set and a nestohedron, which play a fundamental role in our work.

\begin{defi}\label{nest}
A set $B\subseteq 2^{[n+1]}$ is called a {\emph{building set}} on $[n+1]$ if the following two conditions hold:\\
1) For each $i\in [n+1]$: $\{i\}\in B$;\\
2) If $S_{1}, S_{2}\in B$ and $S_{1}\cap S_{2}\neq\varnothing$, then $S_{1}\cup S_{2}\in B$.

A building set $B$ is called {\emph{connected}} if $[n+1]\in B$.
A \emph{nestohedron} $P_B$ on a (connected) building set $B$ on $[n+1]$ is an $n$-dimensional simple convex polytope which is a Minkowski sum of simplices in $\mathbb{R}^{n+1}$ with a standard basis $\{e_{1},\ldots,e_{n+1}\}$:
$$
P_{B}=\sum\limits_{S\in B\backslash [n+1]}\,\conv(e_{i}|\,i\in S).
$$ 
\end{defi}

In what follows we always consider combinatorially equivalent polytopes as equal. The next statement is well known and is a direct consequence of the above definition.

\begin{coro}\label{BuildSetProd}
Suppose $B$ is a union of connected building sets $B_{1}$ on $[n_{1}+1]$ and $B_{2}$ on $[n_{2}+1]$, $[n_{1}+1]\cap [n_{2}+1]=\varnothing$. Then $P_{B}=P_{B_1}\times P_{B_2}$. 
\end{coro}

\begin{exam}
1) Consider $B_{\Delta}=\{\{i\},[n+1]|1\leq i\leq n+1\}$. Then $P_{B_\Delta}=\Delta^n$;
2) Consider $B_{\Box}=\{\{1,\ldots,i\},\{i\}|1\leq i\leq n+1\}$. Then $P_{B_{\Box}}=I^n$. By Corollary~\ref{BuildSetProd} $\hat{B}=\{\{i\},\{2j-1,2j\}|\,1\leq i\leq n+1,1\leq j\leq n\}$ is also a building set for $I^n$. 
\end{exam}

\begin{rema}
The above examples show that there exist different building sets $B_{1}$ and $B_{2}$ such that $P_{B_1}=P_{B_2}$. 
\end{rema}

\begin{defi}\label{Family}
A \emph{family of polytopes} $\mathcal F$ is a set of polytopes such that for any $n\geq 0$ the subset $\mathcal F_n$ of its $n$-dimensional elements is nonempty and finite. A family $\mathcal F$ is called a \emph{line family} if each $\mathcal F_i$ consists of exactly one polytope.
\end{defi}

The following family of polytopes introduced by Carr and Devadoss~\cite{CD} in the framework of the theory of Coxeter complexes consists of flag nestohedra and, therefore, by a result of Buchstaber and Volodin~\cite{BV}, these polytopes can be realized as 2-truncated cubes.  

\begin{defi}\label{GA}
\textit{A graphical building set} $B(\Gamma)$ for a (simple) graph $\Gamma$ on the vertex set $[n+1]$ consists of such $S$ that the induced subgraph $\Gamma_{S}$ on the vertex set $S\subset [n+1]$ is a connected graph.\\
Then $P_{\Gamma}=P_{B(\Gamma)}$ is called a \emph{graph-associahedron}.
\end{defi}

\begin{exam}
The following families of graph-associahedra, which will play an important role in our work, are of particular interest in convex geometry, combinatorics and representation theory: 
\begin{itemize}
\item $\Gamma$ is a complete graph on $[n+1]$.\\
 Then $P_{\Gamma}=Pe^n$ is a \textit{permutohedron}, see Figure~\ref{permfig}.
\begin{figure}[h]
\includegraphics[scale=0.5]{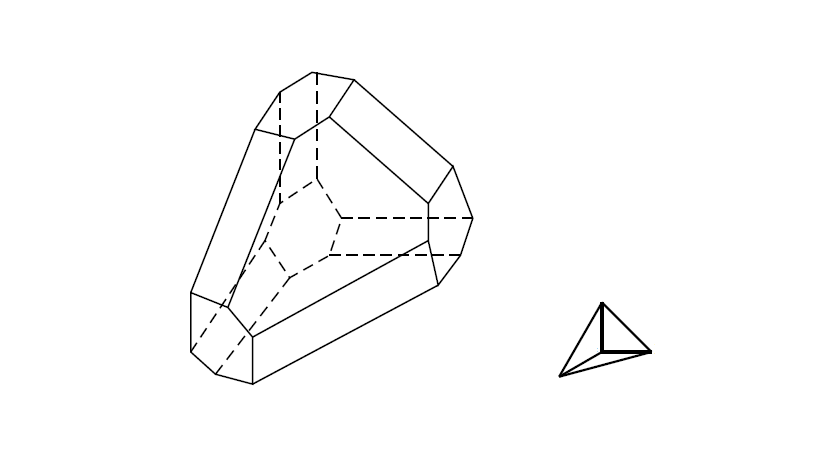}
\caption{3-dimensional permutohedron and the
corresponding graph.}\label{permfig}
\end{figure}

\item $\Gamma$ is a stellar graph on $[n+1]$.\\
 Then $P_{\Gamma}=St^n$ is a \textit{stellahedron}, see Figure~\ref{stelfig}.
\begin{figure}[h]
\includegraphics[scale=0.5]{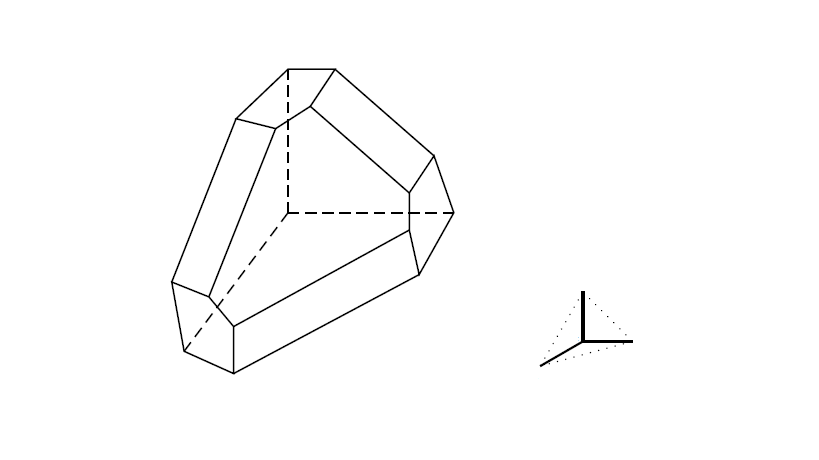}
\caption{3-dimensional stellahedron and the
corresponding graph.}\label{stelfig}
\end{figure}

\item $\Gamma$ is a cycle graph on $[n+1]$.\\
 Then $P_{\Gamma}=Cy^n$ is a \textit{cyclohedron} (or Bott-Taubes polytope~\cite{BT}), see Figure~\ref{cyclfig}.
\begin{figure}[h]
\includegraphics[scale=0.5]{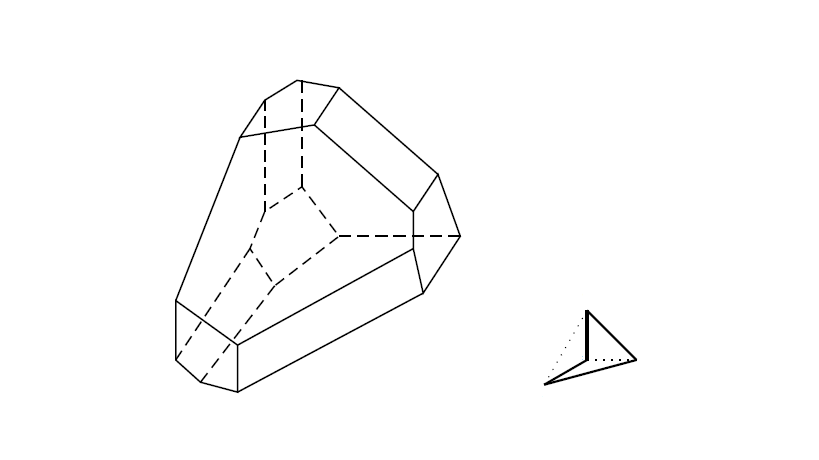}
\caption{3-dimensional cyclohedron and the
corresponding graph.}\label{cyclfig}
\end{figure}

\item $\Gamma$ is a chain graph on $[n+1]$.\\
 Then $P_{\Gamma}=As^n$ is an \textit{associahedron} (or Stasheff polytope~\cite{S}), see Figure~\ref{assfig}.
\begin{figure}[h]
\includegraphics[scale=0.5]{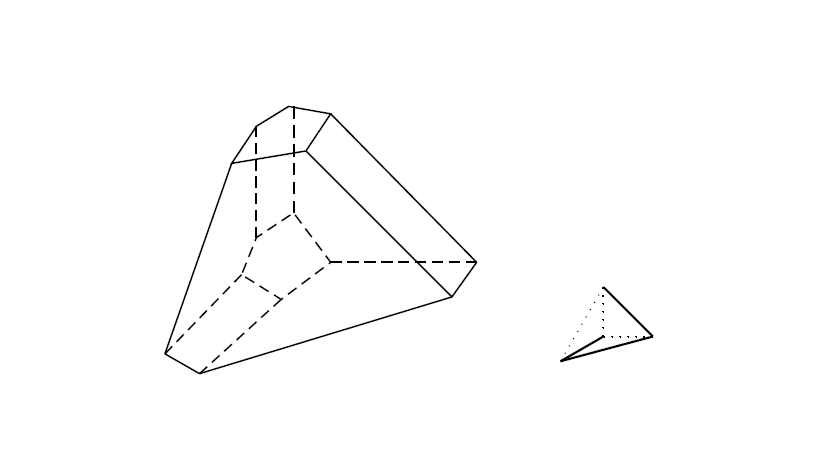}
\caption{3-dimensional associahedron and the
corresponding graph.}\label{assfig}
\end{figure}

\end{itemize}
\end{exam}

\begin{exam}
The families of simplices $\Delta=\{\Delta^n|\,n\geq 0\}$, cubes $\Box=\{I^n|\,n\geq 0\}$, associahedra $As=\{As^n|\,n\geq 0\}$, permutohedra $Pe=\{Pe^n|\,n\geq 0\}$, cyclohedra $Cy=\{Cy^n|\,n\geq 0\}$, and stellahedra $St=\{St^n|\,n\geq 0\}$ are line families.
\end{exam}

\section{Families and the differential ring of polytopes}

There have been obtained a number of results on remarkable families of polytopes, see~\cite{Post05, PRW06}. The theory of families of polytopes that uses the notion of combinatorics of a family was started in the work of Buchstaber~\cite{B}. The notion of a bigraded differential ring of polytopes, which we are going to introduce now, turned out to be of particular importance in this theory.

\begin{defi}
The {\emph{bigraded ring of polytopes}} is $\mathcal P=\oplus\mathcal P_{n}$, where $\mathcal P_{n}=\oplus\mathcal P_{n,k}$. Here $\mathcal P_{n,k}$ is a finitely generated free Abelian group, whose generators correspond to simple $n$-dimensional polytopes $P^n$ with $m$ facets and $k=m-n$. The two operations defined in $\mathcal P$ are: the addition corresponds to a disjoint union of polytopes, and the multiplication corresponds to a Cartesian product of polytopes. Obviously, zero element is an empty set $P^{-1}$ and the unit is a point $P^0$. Note that if $a\in\mathcal P_{n,k}$ and $b\in\mathcal P_{n',k'}$, then $ab\in\mathcal P_{n+n',k+k'}$.

This ring $\mathcal P$ is a differential ring up to the first grading with respect to the following operator $d:\mathcal P\rightarrow\mathcal P$: it maps a polytope to a disjoint union of its facets and thus $d^{2}(P^n)\neq 0$ if $n>1$. Note that $d:\,\mathcal P_{n}\to\mathcal P_{n-1}$.
\end{defi}

\begin{rema}
Similarly, one can also consider a ring $\mathcal{P}^O=\oplus\mathcal P^{O}_{n}, \mathcal P^{O}_{n}=\oplus\mathcal P^{O}_{n,k}$ generated by the set of all oriented convex polytopes. An operator $d^O$ is defined on this ring, which maps an oriented polytope $P$ to a disjoint union of its facets $\{F_i\}$ with orientations prescribed before, in such a way that $F_i$ is taken with ``$+$'' if its orientation coincides with that of $F_i$, induced by the orientation of $P$, and with ``$-$'' otherwise. 
In this case $(d^O)^2 = 0$ and 
\[
d^O(P_1\times P_2) = (d^OP_1)\times P_2 \cup (-1)^{\dim P_1}P_1\times d^OP_2.
\] 
\end{rema}

\begin{prob}
Compute the cohomology ring $H^*[\mathcal{P}^O, d^O]$.
\end{prob}

In this paper we work only with $(\mathcal P,d)$.  
To exclude families with evident combinatorics that are not interesting to us, we introduce the following notions.

\begin{defi}
We call a line family $\mathcal F$ \emph{n-reducible} if any $P\in\mathcal F_k$ with $k\geq n$ is a product of polytopes of positive dimensions. If $\mathcal F$ is not $n$-reducible for any $n>1$, then $\mathcal F$ will be refered to as \emph{irreducible}.
\end{defi}

Note that among the above examples of line families cubes are 2-reducible and other families are irreducible.

Among all convex polytopes $P$ simple polytopes are characterized by the following formula, which holds only for them:
$$
F(dP) = \frac{\partial}{\partial t} F(P),
$$
where $F(P) = \alpha^n+f_{n-1}\alpha^{n-1}t+\ldots+f_0t^n$.
Here $n=\dim P$ and $f_k$ denotes the number of $k$-dimensional faces.

In the theory of simple polytopes, alongside with $F$-polynomial it is convenient to use $H$-polynomials: $H(s,t)=F(s-t,t)$.
For any simple $P$ the classical Dehn-Sommerville theorem on the $f$-vector of $P$ is equivalent to the identity $H(\alpha,t)=H(t,\alpha)$.

For the ring $(\mathcal P,d)$ the following relation holds for the $H$-polynomial:
$$
H(dP)=\partial H(P),
$$
where $\partial=\frac{\partial}{\partial s} + \frac{\partial}{\partial t}$.
Thus, $\partial$ is a linear differential operator that maps symmetric polynomials to symmetric polynomials. Moreover, polynomials $F$ and $H$ provide us with homomorphisms of differential rings:
$$
F:\,\mathcal P\to\mathbb{Z}[\alpha,t]
$$
and
$$
H:\,\mathcal P\to\mathbb{Z}[s,t],
$$
which send $d$ to $\frac{\partial}{\partial t}$ in the case of $F$, and send $d$ to $\partial$ in the case of $H$.

Therefore, formulae for the values of the differential $d$ on families of polytopes transform to partial differential equations.
In a series of works~\cite{B, BK, BV2} it was proved that for several important families of polytopes these differential equations have general analytical solutions (see also~\cite[\S1.7, 1.8]{TT}). Note that~\cite{B} was motivated by the work of Buchstaber and Koritskaya~\cite{BK}, where it was shown that generating series of $H$-polynomials for the associahedra family $As$ satisfies the classical quasilinear E.Hopf equation.
 
Suppose $P=P_B$ is a nestohedron on a connected building set $B\subseteq 2^{[n+1]}$, $n\geq 2$ (see Definition~\ref{nest}). 

Then the following important result holds.

\begin{lemm}[\cite{FS}]\label{NestBound}
The formula for the boundary operator in the ring of polytopes holds:
$$
dP_{B}^n=\sum\limits_{S\in B\backslash [n+1]}\,P_{B|_{S}}\times P_{B/S},\eqno (2)
$$
where the \emph{restriction} building set $B|_{S}$ is defined as an induced connected building set on the vertex set $S\in B$ and the \emph{contraction} of $S$ from $B$ building set $B/S$ is defined as $\{T\subset [n+1]\backslash S\text{and}\,T\in B,\text{or}\;T\sqcup S\in B\}$.
\end{lemm}

\begin{defi}
Suppose $\mathcal F$ is a family of polytopes. Then define $\mathcal P_\mathcal F=\oplus_{n\geq 0}\mathcal {P}_{F,n}$ to be a graded subring of $\mathcal P$ which is generated over $\mathbb Z$ by all $P^n\in F_{n}$. $\mathcal F$ is called a {\emph{differential family}} (or, for short, a {\emph{$d$-family}}) if 
$\mathcal P_{F}$ is a differential ring, i.e. $d:\,\mathcal P_{F,n}\rightarrow\mathcal P_{F,n-1}$.
A line family which is a $d$-family is called a {\emph{d-line family}}.
\end{defi}

Families of simplices, cubes, associahedra, permutohedra are $d$-families.
Moreover, they are also $d$-line families. However, in general, a (line) family is not necessarily a $d$-family. This motivates the following definition.

\begin{defi}
A \emph{$d$-closure} of $\mathcal F$ is a minimal extension of $\mathcal F$ in $\mathcal P$ to a $d$-family.
\end{defi}

The next statement follows directly from Lemma~\ref{NestBound}.

\begin{coro}\label{ClosurePolytopes}
A $d$-closure of any family of nestohedra contains all the nestohedra $P_{B_{n}|_S}$ and $P_{B_{n}/S}$ alongside with $P_{B_n}$ for all $S\in B_{n}\backslash [n+1]$. 
\end{coro}

\begin{exam}
It is known, see~\cite{TT}, that a $d$-closure of the family $Cy$ of cyclohedra is $As\cup Cy$ and a $d$-closure of the family $St$ of stellahedra is $St\cup Pe$.
\end{exam}

The above examples motivate the next definition.

\begin{defi}\label{complex}
We say that a family $\mathcal F$ has {\emph{complexity}} $k\geq 1$ if its $d$-closure is a union of $k$ line families.
\end{defi}

In particular, the above examples show that associahedra and permutohedra families both have complexity 1, cyclohedra and stellahedra families both have complexity 2.

One gets directly from  Definition~\ref{nest} that a combinatorial type of an $n$-dimensional nestohedron $P_{B_n}$ is determined by the structure of the building set $B_n$ on $[n+1]$. Thus, a family $\mathcal F$ of nestohedra is defined by a sequence of building sets $\{B_n, n\geq 1\}$.

Among the important problems of combinatorics of families there is the following one, see Problem IV above:

\begin{prob}
To describe and construct families of polytopes with finite complexity.
\end{prob}

In the introduction we mentioned a new line family of flag nestohedra $\mathcal P_{Mas}$ that will be determined in Definition~\ref{familyP} below. One of our main results is the construction of this family and that $\mathcal P_{Mas}$ has complexity 4, see Theorem~\ref{boundary}. 

In our nearest subsequent publication we will continue the study of the class of families of polytopes having the DFPM property. In particular, we will use Corollary~\ref{ClosurePolytopes} to solve the above problem in the case of nestohedra families. We will also construct new families of polytopes with finite complexity and with nontrivial Massey products.

\section{Two-parametric generating series for nestohedra families and P.D.E.}

To proceed we firstly recall a construction due to Erokhovets. This construction has already found a number of applications among which is~\cite[Proposition 1.5.23]{TT}. We will make use of this construction when studying nestohedra families with nontrivial Massey products.

\begin{constr}[{see~\cite[Construction 1.5.19]{TT}}]
Suppose $B$ is a connected building set on $[n+1]$ and $B_{i}$ is a connected building set on $[k_i]$ for $1\leq i\leq n+1$. Then a {\emph{substitution of building sets}} is a connected building set $B(B_{1},\ldots,B_{n+1})$ on $[k_{1}]\sqcup\ldots\sqcup [k_{n+1}]=[k_{1}+\ldots+k_{n+1}]$, consisting of elements $S_{i}\in B_{i}$ and $\sqcup_{i\in S}\,[k_i]$, where $S\in B$.
\end{constr}

Next we introduce a partially defined operation on the set of connected building sets.

\begin{constr}
Suppose $B_{1}$ and $B_{2}$ are connected building sets on $[n+1]$ and $B_{1}\cap B_{2}=B_{\Delta}$. If a subset $B_{1}\cup B_{2}$ of $2^{[n+1]}$ is a building set, we say that a {\emph{sum of building sets}} $B_{1}+B_{2}$ is defined and equals $B_{1}\cup B_{2}$ as set of subsets of $[n+1]$.
\end{constr}

\begin{prop}\label{sumbuildset}
Suppose $B_{1},B_{2}$ are connected building sets on $[n+1]$ and the following condition holds: for any $S_{i}\in B_{i},i=1,2$ with $S_{1}\cap S_{2}\neq\varnothing$ their union $S_{1}\cup S_{2}\in B_{1}\cup B_{2}$. Then $B_{1}+B_{2}$ is defined.
\end{prop}

Let us consider the following sets of subsets in $[n+1]$ for $n\geq 2$:
$$
B_{1}(P,n)=\{\{i\}|\,1\leq i\leq n+1\},
$$
$$
B_{2}(P,n)=\{\{1,2,i_{1},\ldots,i_{k}\}|\,3\leq i_{1}<\ldots<i_{k}\leq n+1,0\leq k\leq n-1\},
$$
$$
B_{3}(P,n)=\{\{1,j_{1},\ldots,j_{p}\}|\,3\leq j_{1}<\ldots<j_{p}\leq n,1\leq p\leq n-2\}.
$$

\begin{prop}\label{Pbuildset}
The sets $B_{1}(P,n)\cup B_{2}(P,n)$ and $B_{1}(P,n)\cup B_{3}(P,n)\cup \{[n+1]\}$ are connected building sets on $[n+1]$ and their sum 
$$
B(P,n)=(B_{1}(P,n)\cup B_{2}(P,n))+(B_{1}(P,n)\cup B_{3}(P,n)\cup\{[n+1]\})
$$
is defined. Moreover, a nestohedron $P_{B(P,n)}$ is flag.
\end{prop}
\begin{proof}
The first statement follows from the definition of a building set. To prove the second one it sufficies to observe that $(B_{1}(P,n)\cup B_{2}(P,n))\cap(B_{1}(P,n)\cup B_{3}(P,n)\cup\{[n+1]\})=B_{\Delta}$ and, if $S_{1}\in B_{1}(P,n)\cup B_{2}(P,n)$, $S_{2}\in B_{1}(P,n)\cup B_{3}(P,n)\cup \{[n+1]\}$ with $S_{1}\cap S_{2}\neq\varnothing$, then $S_{1},S_{2}\in B_{\Delta}$. Therefore, by Proposition~\ref{sumbuildset} the second statement holds. It is easy to see that every element of $B_{2}(P,n)\cup B_{3}(P,n)$ can be represented as $S_{1}\sqcup S_{2}$ with $S_{1},S_{2}\in B(P,n)$, which finishes the proof. 
\end{proof}

\begin{defi}\label{familyP}
We denote by $\mathcal P_{Mas}=\{P^{n}_{Mas},n\geq 0\}$ a family of flag nestohedra such that $P^0_{Mas}$ is a point, $P^1_{Mas}$ is a segment and $P^{n}_{Mas}=P_{B(P,n)}$ for $n\geq 2$. 
\end{defi}

Note that $P^2_{Mas}$ is a square. The 3-dimensional flag nestohedron $P^3_{Mas}$ is shown in Figure~\ref{3dim2cube} (visible facets are labeled in bold).

\begin{figure}[h]
\includegraphics[scale=0.75]{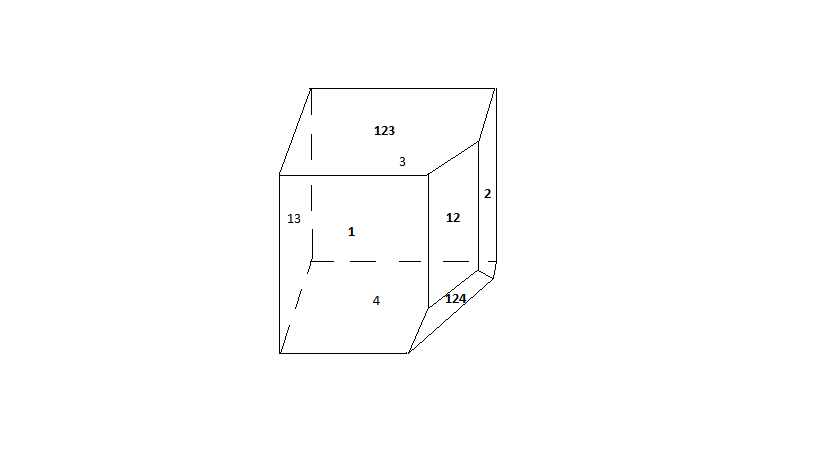}
\caption{The 3-dimensional polytope from the $\mathcal P_{Mas}$ family}
\label{3dim2cube}
\end{figure}

Now we are going to apply the theory of the differential ring of polytopes $(\mathcal P,d)$ to the family of flag nestohedra $\mathcal P_{Mas}$. Using Lemma~\ref{NestBound} for general nestohedra we compute the values of the boundary operator $d$ on polytopes from $\mathcal P_{Mas}$, as well as on flag nestohedra that arise in $dP$ for $P\in\mathcal P_{Mas}$. 

Let us consider four series of flag nestohedra: $\mathcal P_{Mas}$, permutohedra $Pe$, stellahedra $St$, and a certain family of graph-associahedra $\mathcal{P}_{\Gamma}$. Now we introduce the latter family.

Consider the following sets of subsets in $[n+1]$:
$$
B_{1}(\Gamma,n)=\{\{i\}|\,1\leq i\leq n+1\},
$$
$$
B_{2}(\Gamma,n)=\{\{i_{1},\ldots,i_{k}\}|\,1\leq i_{1}<\ldots<i_{k}\leq n,2\leq k\leq n\},
$$
$$
B_{3}(\Gamma,n)=\{\{1,j_{1},\ldots,j_{p},n+1\}|\,2\leq j_{1}<\ldots<j_{p}\leq n,0\leq p\leq n-1\}.
$$

\begin{prop}
The sets $B_{1}(\Gamma,n)\cup B_{2}(\Gamma,n)\cup \{[n+1]\}$ and $B_{1}(\Gamma,n)\cup B_{3}(\Gamma,n)$ are connected building sets on $[n+1]$ and their sum 
$$
B(\Gamma,n)=(B_{1}(\Gamma,n)\cup B_{2}(\Gamma,n)\cup \{[n+1]\})+(B_{1}(\Gamma,n)\cup B_{3}(\Gamma,n))
$$
is defined. Moreover, a nestohedron $P^{n}=P_{B(\Gamma,n)}$ is flag.
\end{prop}
\begin{proof}
Analogous to the proof of Proposition~\ref{Pbuildset}.
\end{proof}

\begin{defi}\label{polytopesPG}
We denote by $\mathcal P_{\Gamma}=\{P_{\Gamma}^{n},n\geq 0\}$ a family of flag nestohedra such that $P_{\Gamma}^0$ is a point, $P_{\Gamma}^1$ is a segment and $P_{\Gamma}^{n}=P_{B(\Gamma,n)}$ for $n\geq 2$.
\end{defi}

\begin{rema}
One can easily see that $B(\Gamma,n)$ coincides with a graphical building set $B(\Gamma_n)$, where $\Gamma_n$ is a simple graph on $n+1$ vertices which consists of a complete graph $K_n$ on $[n]$ and a segment which joins the vertices $\{n\}$ and $\{n+1\}$. 
\end{rema}

To compute the boundary operator on our families of flag nestohedra we use the following formulae for the boundary operator $d$ on $Pe$ and $St$ families, see~\cite{B}:
$$
dPe^n=\sum\limits_{s=0}^{n-1}\binom{n+1}{s+1}Pe^{s}\times Pe^{n-s-1},
$$
and
$$
dSt^n=nSt^{n-1}+\sum\limits_{s=0}^{n-1}\binom{n}{s}St^{s}\times Pe^{n-s-1}.
$$

\begin{lemm}\label{lemmboundary}
Operator $d$ on the graph-associahedra family $\mathcal P_{\Gamma}$ is given by the formula:
$$
dP^{n}_{\Gamma}=Pe^{n-1}+\sum\limits_{s=0}^{n-2}\binom{n-1}{s+1}Pe^{s}\times P^{n-s-1}_{\Gamma}+\sum\limits_{s=0}^{n-1}\binom{n-1}{s}Pe^{s}\times Pe^{n-s-1}+
$$
$$
+\sum\limits_{s=0}^{n-2}\binom{n-1}{s}P^{s+1}_{\Gamma}\times Pe^{n-s-2},
$$
In particular, $\mathcal P_{\Gamma}$ has complexity 2 (see Definition~\ref{complex}).
\end{lemm}
\begin{proof}
From the formula for the boundary of a nestohedron in Lemma~\ref{NestBound} it follows that in the case of the $\mathcal{P}_{\Gamma}$ family the following identities hold:
$$
B(\Gamma,n)/\{1,i_{1},\ldots,i_k,n+1\}=B(Pe^{n-k-2}), 2\leq i_{1}<\ldots<i_{k}\leq n,
$$
$$
B(\Gamma,n)/\{1,i_{1},\ldots,i_{k}\}=B(Pe^{n-k-1}), 2\leq i_{1}<\ldots<i_{k}\leq n,
$$
$$
B(\Gamma,n)/\{j_{1},\ldots,j_{p}\}=B(\Gamma,n-p), 2\leq j_{1}<\ldots<j_{p}\leq n
$$
and for the corresponding induced building sets one gets:
$$
B(\Gamma,n)|_{\{1,i_{1},\ldots,i_{k},n+1\}}=B(\Gamma,k+1),
$$
$$
B(\Gamma,n)|_{\{1,i_{1},\ldots,i_{k}\}}=B(Pe^k),
$$
$$
B(\Gamma,n)|_{\{j_{1},\ldots,j_{p}\}}=B(Pe^{p-1}).
$$
Finally, we note that by definition of contraction operation for building sets one has: $B(\Gamma,n)/\{n+1\}=B(Pe^{n-1})$.

Using the formula for the boundary of a nestohedron and the above formulae, one immediately obtains:
$$
dP^{n}_{\Gamma}=Pe^{n-1}+\sum\limits_{s=0}^{n-2}\binom{n-1}{s+1}Pe^{s}\times P^{n-s-1}_{\Gamma}+\sum\limits_{s=0}^{n-1}\binom{n-1}{s}Pe^{s}\times Pe^{n-s-1}+
$$
$$
+\sum\limits_{s=0}^{n-2}\binom{n-1}{s}P^{s+1}_{\Gamma}\times Pe^{n-s-2}.
$$
\end{proof}

\begin{theo}\label{boundary}
The boundary operator $d$ on $\mathcal P_{Mas}$ is given by the formula:
$$
dP^n_{Mas}=2St^{n-1}+(n-2)P^{n-1}_{Mas}+\sum\limits_{s=0}^{n-2}\binom{n-2}{s}St^{s}\times P^{n-s-1}_{\Gamma}+
$$
$$
+\sum\limits_{s=0}^{n-2}\binom{n-2}{s}St^{s+1}\times Pe^{n-s-2}+\sum\limits_{s=0}^{n-3}\binom{n-2}{s}P^{s+2}_{Mas}\times Pe^{n-s-3}.
$$
In particular, $\mathcal P_{Mas}$ has complexity 4 (see Definition~\ref{complex}).
\end{theo}
\begin{proof}
It is easy to verify that the following identities hold:
$$
B(P,n)/\{1\}=B(\Gamma,n-1), B(P,n)/\{2\}=B(P,n)/\{n+1\}=B(St^{n-1}), 
$$
$$
B(P,n)/\{3\}=\ldots=B(P,n)/\{n\}=B(P,n-1),
$$
$$
B(P,n)/\{1,2,S\}=B(Pe^{n-2-|S|}), B(P,n)|_{\{1,2,S\}}=B(St^{|S|+1}),
$$
$$
B(P,n)/\{1,2,n+1,S\}=B(Pe^{n-3-|S|}), B(P,n)|_{\{1,2,n+1,S\}}=B(P,|S|+2),
$$
$$
B(P,n)/\{1,S\}=B(\Gamma,n-1-|S|), B(P,n)|_{\{1,S\}}=B(St^{|S|}),
$$
for $S\subset\{3,4,\ldots,n\}$.

The result follows now by applying Lemma~\ref{NestBound}.
\end{proof}

To find generating series and the corresponding differential equations for our families of flag nestohedra set
$$
Pe(x)=\sum\limits_{s=0}^{\infty}Pe^{s}\frac{x^{s+1}}{(s+1)!},\quad St(x)=\sum\limits_{s=0}^{\infty}St^{s}\frac{x^s}{s!}
$$
and
$$
P_{\Gamma}(x)=\sum\limits_{s=0}^{\infty}P^{s+1}_{\Gamma}\frac{x^s}{s!},\quad
P_{Mas}(x)=\sum\limits_{s=0}^{\infty}P^{s+2}_{Mas}\frac{x^{s+2}}{s!}.
$$

From Theorem~\ref{boundary} we can now deduce the following equations for the generating series of the above families of flag nestohedra.
Firstly, recall that:
$$
dPe(x)=Pe^2(x),
$$
and
$$
dSt(x)=(x+Pe(x))St(x).
$$

In their studies of generating series for classical polynomials~\cite{B-Kh} Buchstaber and Kholodov used a fruitful idea of introducing generating series with structure sequences $\{a_n\}$. It allowed them to treat generating series for different classes of polynomials in a univeral way, changing the structure constants $a_n$.

We are going to apply a similar approach to the generating series of polytope families. For a generating series $Q(x)=\sum\limits_{n=0}^{\infty}a_{n}Q^{n}x^{n+n_0}$ set $dQ(x)=\sum\limits_{n=0}^{\infty}a_{n}d(Q^{n})x^{n+n_0}$. 

\begin{theo}\label{1paramGenSer}
The following identities take place:
$$
dP_{\Gamma}(x)=2Pe(x)P_{\Gamma}(x)+(1+\frac{d}{dx}Pe(x))\frac{d}{dx}Pe(x),
$$
$$
dP_{Mas}(x)=(x+Pe(x))P_{Mas}(x)+x^{2}(2\frac{d}{dx}St(x)+St(x)P_{\Gamma}(x)+\frac{d}{dx}St(x)\frac{d}{dx}Pe(x)).
$$
\end{theo}
\begin{proof}
Direct calculation using Lemma~\ref{lemmboundary} and Theorem~\ref{boundary}.
\end{proof}

Now, following~\cite{BE}, we introduce another parameter and study 2-parametric generating series for our families of flag nestohedra.

\begin{defi}\label{2paramSer}
For a family $\mathcal P=\{P^{n}|\,n\geq 0\}$ of simple polytopes consider its generating series $P(x)=\sum\limits_{n=0}^{\infty}a_{n}P^{n}x^{n+n_0}$ and the following formal series ($P^{n}\in\mathcal P$):
$$
Q(P^{n};q)=\sum\limits_{k=0}^{n}(d^{k}P)\frac{q^k}{k!}.
$$
Then define the {\emph{two-parameter extension}} of the generating series of $\mathcal P$ to be:
$$
P(q,x)=\sum\limits_{n=0}^{\infty}a_{n}Q(P^{n};q)x^{n+n_0}.
$$
We have: $Q(P^{n};0)=P^{n}$ and $P(0,x)=P(x)$.
\end{defi}

Recall that the 2-parametric generating series $Pe(q,x)$ and $St(q,x)$ are solutions of the following Cauchy problems:
\begin{itemize}
\item[(1)] For the 2-parametric generating series of permutohedra family:
$$
\frac{\partial}{\partial q}Pe(q,x)=Pe^2(q,x), Pe(0,x)=Pe(x).
$$
The solution has the following form:
$$
Pe(q,x)=\frac{Pe(x)}{1-qPe(x)};
$$
\item[(2)] For the 2-parametric generating series of stellahedra family:
$$
\frac{\partial}{\partial q}St(q,x)=(x+Pe(q,x))St(q,x), St(0,x)=St(x).
$$
The solution has the following form:
$$
St(q,x)=St(x)\frac{e^{qx}}{1-qPe(x)}.
$$
\end{itemize}
 
\begin{theo}\label{2paramGenSer}
The following statements hold.
\begin{itemize}
\item[(1)] 2-parametric generating series $P_{\Gamma}(q,x)$ is a solution of the following Cauchy problem:
$$
\frac{\partial}{\partial q}P_{\Gamma}(q,x)=2Pe(q,x)P_{\Gamma}(q,x)+\frac{\partial}{\partial x}Pe(q,x)(1+\frac{\partial}{\partial x}Pe(q,x)), P_{\Gamma}(0,x)=P_{\Gamma}(x);
$$
\item[(2)] 2-parametric generating series $P_{Mas}(q,x)$ is a solution of the following Cauchy problem:
$$
\frac{\partial}{\partial q}P_{Mas}(q,x)=(x+Pe(q,x))P_{Mas}(q,x)+x^{2}[2\frac{\partial}{\partial x}St(q,x)+St(q,x)P_{\Gamma}(q,x)+
$$
$$
+\frac{\partial}{\partial x}St(q,x)\frac{\partial}{\partial x}Pe(q,x)], P_{Mas}(0,x)=P_{Mas}(x).
$$
\end{itemize}
\end{theo}
\begin{proof}
Direct calculation using Theorem~\ref{1paramGenSer}. 
\end{proof}

We shall give a complete solution of the above Cauchy problems, analyze their behavior, and interpret the consequences of these solutions in terms of combinatorics of the underlying families in our nearest subsequent publication. 


\section{Flag nestohedra and Massey products}

We first recall the construction of a line family of 2-truncated cubes $\mathcal Q=\{Q^n|\,n\geq 0\}$, for which $\mathcal Z_{Q^{n}}$ has a strictly defined nontrivial $n$-fold Massey product in cohomology, $n\geq 2$.

\begin{defi}[\cite{L2}]\label{2truncMassey}
Set $Q^0$ to be a point and $Q^1\subset\R^1$ to be a segment $[0,1]$. Suppose $I^{n}=[0,1]^n, n\geq 2$ is an $n$-dimensional cube with facets $F_{1},\ldots,F_{2n}$, such that $F_{i},1\leq i\leq n$ contains the origin 0, $F_{i}$ and $F_{n+i}$ for $1\leq i\leq n$ are parallel. 
Then its face ring is the following one:
$$
\ko[I^n]=\ko[v_{1},\ldots,v_{n},v_{n+1},\ldots,v_{2n}]/I_{I^n},
$$
where $I_{I^n}=(v_{1}v_{n+1},\ldots,v_{n}v_{2n})$.

Consider the polynomial ring
$$
\mathbb{Z}[v_{1},\ldots,v_{2n},w_{k',n+k'+i'}|\,1\leq i'\leq n-2, 1\leq k'\leq n-i']
$$
and the following square free monomial ideal
$$
I=(v_{k}v_{n+k+i},w_{k',n+k'+i'}v_{n+k'+l},w_{k',n+k'+i'}v_{p},w_{k',n+k'+i'}w_{k'',n+k''+i''}),
$$
in the above ring, where $v_{j}$ corresponds to $F_{j}$ for $1\leq j\leq 2n$, and 
$$
0\leq i\leq n-2, 1\leq k\leq n-i, 1\leq i',i''\leq n-2, 1\leq k'\leq n-i', 
$$
$$
1\leq k''\leq n-i'', 1\leq p\neq k'\leq k'+i', 0\leq l\neq i'\leq n-2, 
$$
$$
k'+i'=k''\,\text{or }k''+i''=k'.
$$

Let us define $Q^n\subset\R^n$ to be a simple polytope such that $I_{Q^n}=I$. Observe that $Q^n$ has a natural realization as a 2-truncated cube and, furthermore, its combinatorial type does not depend on the order of face truncations of $I^n$.
\end{defi}

The next result on higher Massey products in cohomology of moment-angle manifolds holds.

\begin{theo}[{\cite{L2}}]\label{mainMassey}
Let $\alpha_i\in H^{3}(\mathcal Z_{Q^n})$ be represented by a 3-cocycle $v_{i}u_{n+i}\in R^{-1,4}(Q^n)$ for $1\leq i\leq n$ and $n\geq 2$. Then all Massey products of consecutive elements from $\alpha_{1},\ldots,\alpha_{n}$ are defined and the whole $n$-product $\langle\alpha_{1},\ldots,\alpha_{n}\rangle$ is nontrivial.
\end{theo}

\begin{prop}\label{Qdfpnm}
The family $\mathcal Q$ is a special geometric direct family with nontrivial Massey products.
\end{prop}
\begin{proof}
Due to Theorem~\ref{mainMassey} above it suffices to prove that $\mathcal Q$ is an algebraic direct family, see Definition~\ref{ADFP}. The latter follows from~\cite[Theorem 3.2]{B-L}. The family $\mathcal Q$ is special due to~\cite[Corollary 3.3]{B-L}.
\end{proof}

We are going to prove that $\mathcal P_{Mas}$ is a line geometric direct family of flag nestohedra with nontrivial Massey products.
The notation for this family is justified by means of the following lemma.

\begin{lemm}\label{FlagNestMassey}
For any $P^{n}_{Mas}\in\mathcal P_{Mas}$ with $n\geq 2$ there exists a strictly defined and nontrivial $n$-fold Massey product in $H^*(\mathcal Z_{P^n_{Mas}})$.
\end{lemm}
\begin{proof}
We first obtain $P^{n}_{Mas}\in\mathcal P_{Mas}$ with $n\geq 2$ as a 2-truncated cube using the iterative procedure of codimension 2 face cuts off from $I^n$ given in~\cite{BV}. In the notation of Definition~\ref{2truncMassey} we identify $F_{i}$ with $\{1,\ldots,i\}$ for $1\leq i\leq n$ and identify $F_{i}$ with $\{i-n+1\}$ for $n+1\leq i\leq 2n$. Then we consecutively cut the following faces from $I^n$:
$$
\{1\}\sqcup\{3\}, \{1,2\}\sqcup\{4\},\ldots,\{1,\ldots,n-1\}\sqcup\{n+1\}\\
$$
$$
\cdots\\
$$
$$
\{1\}\sqcup\{n\}, \{1,2\}\sqcup\{n+1\}.
$$
in order, which is opposite to the inclusion. It is easy to see that $B=B(P,n)$ is a union of the connected building set $B_0$ of $I^n$, the set $B_{1}$ of above subsets of $[n+1]$, and all the subsets of $[n+1]$ which are the unions of nontrivially intersecting elements in $B$, that is, if $S_{1},S_{2}\in B$ and $S_{1}\cap S_{2}\neq\varnothing,S_{1},S_{2}$, then one should have: $S_{1}\cup S_{2}\in B$ by definition of a building set. One has: $S=S_{1}\sqcup S_{2}\in B(P,n)$ with $S_{1},S_{2}\in B_0$ if and only if $S\in B_1$ and, therefore, the full subcomplexes: on the vertex set $[2n]$ in $K_{Q^n}=\partial Q^*$ and on the vertex set $\{\{i\},\{1,2,\ldots,k\}|1\leq i\leq n+1, 2\leq k\leq n\}$ in $K_{P^n_{Mas}}$ are isomorphic. To finish the proof it suffices to apply Theorem~\ref{mainMassey} and Theorem~\ref{BPtheo}.
\end{proof}

\begin{rema}
It is easy to see that $Q^n$ and $P^n_{Mas}$ are combinatorially equivalent if and only if $n\leq 3$; for any $n\geq 1$ one has: $f_{0}(P^n_{Mas})=3\times 2^{n-2}+n-1>f_{0}(Q^n)=\frac{n(n+3)}{2}-1$. Moreover, for any $n>3$ $Q^n$ is a 2-truncated cube, but not a flag nestohedron. Note also that for any $n>2$ $P^n_{Mas}\in\mathcal P_{Mas}$ is a flag nestohedron, but not a graph-associahedron.
\end{rema}


\begin{defi}(\cite{FS})
For a building set $B$ the nerve complex $K_P=\partial P^*$ for $P=P_B$ is called a {\emph{nested set complex}} and is denoted by $N_B$.
\end{defi}

Face poset structure of a nested set complex $N_B$ can be described in terms of the structure of the building set $B$ as follows.
 
\begin{prop}[{\cite[Theorem 1.5.13]{TT}}]\label{Fposet}
Vertices of $N_{B}$ are in one-to-one correspondence with nonmaximal elements in $B$.\\ 
Moreover, a set of vertices corresponding to such elements $S=\{S_{i_1},\ldots, S_{i_k}\}$ forms a simplex in $N_B$ if and only if:
\begin{itemize}
\item[(1)] For any two elements $S_{i_p},S_{i_q}$ with $1\leq p,q\leq k$, either $S_{i_p}\cap S_{i_q}=\varnothing$, or $S_{i_p}$, or $S_{i_q}$;

\item[(2)] If elements of a subset $\{S_{i_{t_1}},\ldots,S_{i_{t_l}}\}\subseteq S,l\geq 2$ are pairwisely disjoint, then $S_{i_{t_1}}\sqcup\ldots\sqcup S_{i_{t_l}}\notin B$.
\end{itemize}
\end{prop}

\begin{lemm}\label{NestSetRestrict}
Suppose $P=P_B$ is a nestohedron on a connected building set $B$ on $[n+1]$. Then for any $S\in B$ the full subcomplex of the nested set complex $N_{B}$ on the vertex set $B|_S$ is combinatorially equivalent to the nested set complex $N_{B|_S}$. 
\end{lemm}
\begin{proof}
By Proposition~\ref{Fposet} it suffices to observe that for pairwisely disjoint elements $S_{i_{t_1}},\ldots,S_{i_{t_l}}\in B_S$ one has: $S_{i_{t_1}},\ldots,S_{i_{t_l}}\subseteq S$ and $S_{i_{t_1}}\sqcup\ldots\sqcup S_{i_{t_l}}\subseteq S$, thus $S_{i_{t_1}}\sqcup\ldots\sqcup S_{i_{t_l}}\notin B$ if and only if $S_{i_{t_1}}\sqcup\ldots\sqcup S_{i_{t_l}}\notin B|_S$.
\end{proof}

\begin{exam}
1. If $P=Pe^n$, then for any $S\in B(P)$ one has: $P_{B|_S}=Pe^{|S|-1}$ and $N_{B|_S}$ is a full subcomplex in $N_{B}$ (both of them are barycentric subdivisions of boundaries of simplices).
2. If $P=As^n$, then for any path graph $S$ on $[i],i\leq n+1$ one has: $P_{B|_S}=As^{|S|-1}$ and $N_{B|_S}$ is a full subcomplex in $N_B=\partial (As^{n})^*$.
\end{exam}

\begin{rema}
Note that the above Lemma does not hold when $B|_S$ is replaced by $(B/S)\cap B$. Indeed, consider $B=B(P,4)$ to be the building set of $P^{4}\in\mathcal P_{Mas}$, $S_{1}={2}$, $S_{2}={4}$ and $S={1,3}$. Then $S\in B$, $S_{1},S_{2}\in B/S$, $(B/S)\cap B=\{\{2\},\{4\}\}$ and $S_{1}\sqcup S_{2}\notin B$. However, $S_{1}\sqcup S_{2}\in B/S$, since $S_{1}\sqcup S_{2}\sqcup S=\{1,2,3,4\}\in B$. We immediately obtain that the full subcomplex on $(B/S)\cap B$ in $N_{B}$ is a 1-simplex, but $N_{B/S}$ consists of two disjoint points.
\end{rema}

Consider a connected building set $B$ on $[n+1]$ and $S\in B\backslash [n+1]$.
Due to~\cite[Construction 2.8]{B-L} we have an induced embedding of moment-angle manifolds $\hat{\phi}_{S}:\,\mathcal Z_{P_{B|_S}}\to\mathcal Z_{P_B}$ and due to Lemma~\ref{NestSetRestrict} and~\cite[Corollary 2.7]{B-L} we have an induced embedding of moment-angle-complexes $j_{S_{*}}:\,\mathcal Z_{N_{B|_S}}\to\mathcal Z_{N_B}$ that has a retraction map, where $N_B=K_{P_B}$ and $N_{B|_S}=K_{P_{B|_S}}$. Combinatorial properties of nestohedra allow us to obtain the following result.

\begin{prop}\label{phiPsi}
Suppose $P_B$ is a nestohedron. Then there is a commutative diagram
$$\begin{CD}
  \mathcal Z_{P_{B|_S}} @>\hat{\phi}_{S}>> \mathcal Z_{P_{B}}\\
  @VVh_{1} V\hspace{-0.2em} @VVh_{2} V @.\\
  \mathcal Z_{N_{B|_S}} @>j_{S_*}>>\mathcal Z_{N_B},
\end{CD}\eqno 
$$
where $h_{1}$ and $h_{2}$ are homeomorphisms. In particular, $\hat{\phi}_{S}$ induces a split epimorphism in cohomology for any $S\in B\backslash [n+1]$.
\end{prop}
\begin{proof}
This follows from~\cite[Corollary 2.14]{B-L}.
\end{proof}

Note that Proposition~\ref{phiPsi} implies the induced embedding of moment-angle manifolds $\hat{\phi}_{S}$ has a retraction map (and thus induces a split ring epimorphism in cohomology) for any $P_B$, not necessarily flag.

Suppose $B$ is a building set, $S\in B$ and denote by $P_S$ the nestohedron $P_{B|_S}$.

\begin{prop}\label{CohomEpi}
The following statements hold.
\begin{itemize}
\item[(1)] If $j:\,N_{B|_S}\hookrightarrow N_B$, then $j^*:\,H^*(\mathcal Z_{P_B})\rightarrow H^*(\mathcal Z_{P_S})$ is a natural split ring epimorphism;
\item[(2)] If a Massey product $\langle\alpha_{1},\ldots,\alpha_{k}\rangle\in H^*(\mathcal Z_{P_S})$ is defined, then there exist $\beta_{t}\in H^*(\mathcal Z_{P_B})$ for $1\leq t\leq k$ such that $j^*(\beta_{t})=\alpha_t$, $\langle\beta_{1},\ldots,\beta_{k}\rangle$ is defined, and $j^*\langle\beta_{1},\ldots,\beta_{k}\rangle=\langle\alpha_{1},\ldots,\alpha_{k}\rangle$. Moreover, if $\langle\alpha_{1},\ldots,\alpha_{k}\rangle$ is nontrivial (resp. strictly defined), then $\langle\beta_{1},\ldots,\beta_{k}\rangle$ is nontrivial (resp. strictly defined).
\end{itemize}
\end{prop}
\begin{proof}
Statement (1) follows directly from Lemma~\ref{NestSetRestrict} and~\cite[Corollary 2.7]{B-L}.
To prove statement (2) note that $N_{B|_S}$ is a full subcomplex in $N_B$ by Lemma~\ref{NestSetRestrict}, thus we can apply (1) to $j^*$ induced by an embedding of building sets $B|_S\rightarrow B$ that gives an embedding of simplicial complexes $j:\,N_{B|_S}\hookrightarrow N_B$. The proof finishes by applying Theorem~\ref{BPtheo}. 
\end{proof}

Using Proposition~\ref{CohomEpi} we can get the following combinatorial condition on a line family $\mathcal F$, under which $\mathcal F$ is a GDFP. 
Namely, the following statement holds.

\begin{prop}\label{NestDirect}
Suppose $\mathcal F=\{P_{B(n)}|\,n\geq 0\}$ is a line family of nestohedra on connected building sets $B(n)$ on $[n+1]$ for $n\geq 0$. If for any $n>2$ there exists an element $S(n)\in B(n)$ with $|S(n)|=n$ such that $P_{B|_{S(n)}}=P_{B(n-1)}$, then $\mathcal F$ is a geometric direct family.\\
In particular, the following nestohedra families are GDFP: cubes $\Box$, permutohedra $Pe$, stellahedra $St$, cyclohedra $Cy$, and associahedra $As$.
\end{prop}
\begin{proof}
By Lemma~\ref{NestSetRestrict} $N_{B|_{S(n)}}$ is a full subcomplex in $N_{B(n)}$ for each $n>2$. On the other hand, the formula for the boundary operator $d$ value on $\mathcal F$ shows that $P_{B|_{S(n)}}\cong P_{B(n-1)}$ is combinatorially equivalent to a facet of $P_{B(n)}$. Due to Proposition~\ref{CohomEpi} (1) the embedding $j:\,N_{B|_{S(n)}}\hookrightarrow N_{B(n)}$ of the full subcomplex into the nested set complex induces a split ring epimorphism in cohomology of their moment-angle-complexes. The proof of the first part of the statement finishes by induction on $n$ and Proposition~\ref{phiPsi}. The rest follows from the explicit formulas for the value of the boundary operator $d$ on those families of nestohedra, see~\cite{B}, and the first part of the statement.
\end{proof}

\begin{exam}
Note that the family of simplices $\Delta$ is not even an ADFP (in particular, Proposition~\ref{NestDirect} is not valid), since for any $n>2$ and any $S\in B(n)\backslash [n+1]$ one has: $P_{B|_S}$ is a point and $(K_P)_J$ is either $K_P$, or a simplex, for any $P\in\Delta$.
\end{exam}

\begin{theo}\label{MasAllOrder}
For $P^{n}_{Mas}\in\mathcal P_{Mas}$ there exists a nontrivial strictly defined $k$-fold Massey product $\langle\alpha_{1}^{n},\ldots,\alpha_{k}^{n}\rangle$ in $H^*(\mathcal Z_{P^n_{Mas}})$ for any $2\leq k\leq n$.
Moreover, for any $2\leq r\leq s$ there exists a natural embedding $j_{r}^{s}:\,N_{B(P,r)}\hookrightarrow N_{B(P,s)}$ such that $(j_{r}^{s})^{*}\langle\alpha_{1}^{s},\ldots,\alpha_{k}^{s}\rangle=\langle\alpha_{1}^{r},\ldots,\alpha_{k}^{r}\rangle$ when $2\leq k\leq r$.
\end{theo}
\begin{proof}
Suppose $B=B(P,n)$. By Theorem~\ref{boundary} for $S_{t}=\{1,2,n+1,I\}$, where $I=\{3,4,\ldots,t\}$ and $2\leq t\leq n-1$ (for $t=2$ set $S_{t}=\{1,2,n+1\}$ and $I=\varnothing$), one has: $P_{B|_{S_t}}\cong P^{t}_{Mas}\in\mathcal P_{Mas}$. Consider a natural embedding of building sets $B|_{S_{t}}\rightarrow B$. Due to Lemma~\ref{NestSetRestrict} the latter map induces an embedding of nested set complexes $j_{t}^{n}:\,N_{B|_{S_t}}\hookrightarrow N_{B}$, that is, there exists
a full subcomplex in $N_{B(P,n)}=K_{P^n_{Mas}}$ isomorphic to $N_{B(P,k)}$ for $2\leq k\leq n-1$, which finishes the proof by applying Lemma~\ref{FlagNestMassey} and Proposition~\ref{CohomEpi}. Since $\mathcal P_{Mas}$ consists of flag polytopes, the rest of the proof could also be obtained applying~\cite[Theorem 2.16]{B-L} instead of Proposition~\ref{CohomEpi}.
\end{proof}

\begin{rema}
Another way to prove the Theorem above is to apply~\cite[Corollary 3.7]{L3} to the induced subcomplex of $N_{B(P,n)}$ on the vertex set $\{\{i\},\{1,2,\ldots,k\}|1\leq i\leq n+1, 2\leq k\leq n\}$ (see the proof of Lemma~\ref{FlagNestMassey}). 
\end{rema}

\begin{theo}\label{DirFamNontrivial}
The family $\mathcal P_{Mas}$ is a special geometric direct family with nontrivial Massey products.
\end{theo}
\begin{proof}
It follows from Proposition~\ref{phiPsi} and Proposition~\ref{CohomEpi} that $\mathcal P_{Mas}$ is a geometric direct family, since conditions (1) and (2) of Definition~\ref{ADFP} hold for $\mathcal P_{Mas}$ by Theorem~\ref{boundary} (see also Proposition~\ref{NestDirect}).
Furthermore, using Lemma~\ref{FlagNestMassey} one immediately obtains that $\mathcal P_{Mas}$ satisfies the condition on nontrivial Massey products from Definition~\ref{DFPM}. Finally, the family $\mathcal P_{Mas}$ is special due to Theorem~\ref{MasAllOrder}. This finishes the proof.
\end{proof}

\begin{prop}\label{MasseyNestohedra}
The following statements hold.
\begin{itemize}
\item[(1)] Suppose $P_{B}$ is a nestohedron on a connected building set $B$ on $[n+1]$. Then if $P_{B|_S}$ for some $S\in B\backslash [n+1]$ has a nontrivial strictly defined $k$-fold Massey product in cohomology, then the same holds for $P_B$;
\item[(2)] Suppose $P_{B}$ is a flag nestohedron on a connected building set $B$ on $[n+1]$. Then if $P_{B/S}$ for some $S\in B\backslash [n+1]$ has a nontrivial strictly defined $k$-fold Massey product in cohomology, then the same holds for $P_B$.
\end{itemize}
\end{prop}
\begin{proof}
The first part of the statement is a direct corollary of Proposition~\ref{phiPsi} above. The second part follows from~\cite[Theorem 2.16]{B-L}, since by Lemma~\ref{NestBound} $P_{B/S}$ is always a face of $P_B$.
\end{proof}

\begin{coro}\label{GAseqMas}
Suppose $\mathcal P$ is one of the following classical families of graph-associahedra: associahedra, permutohedra, cyclohedra, or stellahedra. Then $\mathcal P=\{P^n|\,n\geq 0\}$ is a sequence of flag simple polytopes such that there exists a nontrivial $k$-fold Massey product in $H^*(\mathcal Z_{P^n})$, $n\geq 2$ with $k=2$, or 3. Moreover, existence of a nontrivial $k$-fold Massey product in $H^*(\mathcal Z_{P^n})$ implies existence of a nontrivial $k$-fold Massey product in $H^*(\mathcal Z_{P^l})$ for any $k\geq 3$ and any $l>n$.
\end{coro}
\begin{proof}
In~\cite{L2,L3} it was proved that in cohomology of a moment-angle manifold over $As^3$, $Pe^3$, $Cy^3$, and $St^3$ there exists a strictly defined nontrivial triple Massey product. For any $P^{n}=P_B$ in these graph-associahedra families, $n\geq 4$, there exists an element $S\in B$ on 4 vertices such that $P_{B|_S}$ equals one of the 3-dimensional graph-associahedra above: $As^3$, $Pe^3$, $Cy^3$, or $St^3$. Therefore, Proposition~\ref{MasseyNestohedra} implies our statement.
\end{proof}

We finish this section with a generalization of Lemma~\ref{FlagNestMassey}.

\begin{defi}\label{FamilyF}
Consider a family of flag nestohedra on connected building sets that are obtained from $\mathcal P_{Mas}$ by applying substitution of building sets operation. We denote this family by $\mathcal F_{Mas}$.
\end{defi}

The notation from the previous definition is clearified by the next result.

\begin{theo}\label{NestMassey}
For any $l\geq 2, r\geq 1$, and any set $S=\{n_{i}\geq 2|\,1\leq i\leq r\}$ there is an element $P(S)=P_{B(S)}\in\mathcal F_{Mas}$ and an $l$-connected moment-angle manifold $M(l,S)$ over a multiwedge of $P(S)$ with a strictly defined and nontrivial $n$-fold Massey product in $H^*(M(l,S))$ for all $n\in S$.
\end{theo}
\begin{proof}
Consider the case $l=2$. By Lemma~\ref{FlagNestMassey} for the connected building sets $B(i)=B(P,n_{i}),1\leq i\leq r$ with $n_{i}\in S$ one has: there exists a strictly defined and nontrivial Massey product of order $n_i$ in $H^*(\mathcal Z_{P_{B(i)}})$. Then define a connected building set $B_{1}(S)=B(B(1),\ldots,B(r))$ to be the result of substitution of building sets (see~\cite[Construction 1.5.19]{TT}), where $B=B(P,r-1)$. Then $P_{B_{1}(S)}$ is a flag nestohedron which is combinatorially equivalent to $P^{r-1}_{Mas}\times P^{n_1}_{Mas}\times\ldots\times P^{n_r}_{Mas}$ (see~\cite[Lemma 1.5.20]{TT}) and, moreover, $P_{B_{1}(S)}\in\mathcal F$ with $M(2,S)=\mathcal Z_{P_{B_{1}(S)}}$.

Now suppose $l\geq 3$. Following~\cite[Definition 3.5]{L3} denote by $P(n,s)=P^{n}_{Mas}(J_{n,s})$ the corresponding multiwedge, see Definition~\ref{simpmultwedge}. Then for $s=[\frac{l+1}{2}]$ the nerve complex of $P(n,s)$ is $[\frac{l+1}{2}]$-connected and thus the generalized moment-angle manifold $M(l,n)=\mathcal Z_{{P^n}_{Mas}}^{J_{n,s}}\cong\mathcal Z_{P(n,s)}$ is $l$-connected. By~\cite[Theorem 3.6]{L3} there exists a strictly defined and nontrivial $n$-fold Massey product in $H^*(M(l,n))$.

Consider the product of the above building sets in the ring of building sets, see~\cite[\S 1.7]{TT}
$$
B_{2}(S)=B(1)\cdot\ldots\cdot B(r)
$$
and the ordered set of positive integers:
$$
J(l,S)=(J_{n_{1},s},\ldots,J_{n_{r},s}).
$$
Then let
$$
M(l,S)=\mathcal Z_{P_{B_{2}(S)}}^{J(l,S)}\cong\prod\limits_{i=1}^{r}\mathcal Z_{P(n_{i},s)}.
$$
This generalized moment-angle manifold is $l$-connected as a product of $l$-connected manifolds and has a strictly defined and nontrivial $n$-fold Massey product in $H^*(M(l,S))$ for any $n\in S$.

Finally, one has: $B_{1}(S)\in\mathcal F_{Mas}$ for all nonempty finite subsets $S$ of $\mathbb{N}$ with elements grater than one. Using~\cite[Proposition 1.5.23]{TT}, we can get a {\emph{connected}} building set $B(S)$, which can be obtained iterating substitution of building sets of polytopes in $\mathcal P_{Mas}$, such that $P_{B_{2}(S)}$ is combinatorially equivalent to $P_{B(S)}$. This finishes the proof.
\end{proof}

\begin{rema}
Note that both families of flag nestohedra on connected building sets $\mathcal P_{Mas}\subset\mathcal F_{Mas}$ are countable.
\end{rema}


\end{document}